\documentclass[10pt,oneside,english]{amsart}
\usepackage[T1]{fontenc}
\usepackage[latin9]{inputenc}
\usepackage[a4paper]{geometry}
\usepackage{amsthm}
\usepackage{amstext}
\usepackage{amssymb}

\makeatletter
\numberwithin{equation}{section}
\numberwithin{figure}{section}
  \theoremstyle{plain}
  \newtheorem*{thm*}{\protect\theoremname}
\theoremstyle{plain}
\newtheorem{thm}{\protect\theoremname}
  \theoremstyle{definition}
  \newtheorem{defn}[thm]{\protect\definitionname}
  \theoremstyle{plain}
  \newtheorem{prop}[thm]{\protect\propositionname}
  \theoremstyle{definition}
  \newtheorem{example}[thm]{\protect\examplename}
  \theoremstyle{remark}
  \newtheorem{rem}[thm]{\protect\remarkname}
  \theoremstyle{plain}
  \newtheorem{lem}[thm]{\protect\lemmaname}
  \theoremstyle{plain}
  \newtheorem{cor}[thm]{\protect\corollaryname}

\theoremstyle{definition}
\newtheorem{parn}{}[subsection]

\makeatother

\usepackage{babel}
  \providecommand{\corollaryname}{Corollary}
  \providecommand{\definitionname}{Definition}
  \providecommand{\examplename}{Example}
  \providecommand{\lemmaname}{Lemma}
  \providecommand{\propositionname}{Proposition}
  \providecommand{\remarkname}{Remark}
  \providecommand{\theoremname}{Theorem}
\providecommand{\theoremname}{Theorem}

\begin{document}
\author{Adrien Dubouloz}
\address{Institut de Math\'ematiques de Bourgogne,  9 avenue Alain Savary, 21 078 Dijon, France} \email{adrien.dubouloz@u-bourgogne.fr}
\author{Takashi Kishimoto}
\address{Department of Mathematics, Faculty of Science, Saitama University, Saitama 338-8570, Japan} 
\email{tkishimo@rimath.saitama-u.ac.jp}
 \thanks{The first author was partialy supported by ANR Grant  "BirPol"  ANR-11-JS01-004-01. The second author was supported by a Grant-in-Aid for Scientific Research of JSPS No. 24740003. This work was done during a stay of the first author at Saitama University and a stay of the second author at the Institut de Math\'ematiques de Bourgogne (Dijon). The authors thank these institutions for the hospitality.}

\title{Log-uniruled affine varieties without cylinder-like open subsets}
\begin{abstract}
A classical result of Miyanishi-Sugie and Keel-M${\rm {}^c}$kernan asserts that for smooth affine
surfaces, $\mathbb{A}^{1}$-uniruledness is equivalent to $\mathbb{A}^{1}$-ruledness,
both properties being in fact equivalent to the negativity of the
logarithmic Kodaira dimension. Here we show in contrast that starting
from dimension three, there exists smooth affine varieties which are
$\mathbb{A}^{1}$-uniruled but not $\mathbb{A}^{1}$-ruled.
\end{abstract}
\maketitle

\section*{Introduction}

Complex uniruled projective varieties are nowadays considered through the
Mori Minimal Model Program (MMP) as the natural generalization to higher
dimensions of birationaly ruled surfaces (see e.g. \cite{KM98}). In particular, 
as in the case of ruled surfaces, these are the varieties for which the program
does not yield a minimal model, but a Mori fiber space. These varieties 
are also conjectured to be the natural generalization to higher dimensions
of surfaces of negative Kodaira dimension, the conjecture being in fact established
as long as the abundance conjecture holds true \cite{FA}, hence in particular for smooth threefolds. 
During the past decades, the systematic study of the geometry of rational curves on these varieties
has been the source of many progress in the structure theory for higher
dimensional, possibly singular, projective varieties to which the MMP can be applied. 
The situation is much less clear for non complete varieties, in particular affine ones. 

The natural analogue of ruledness in this context is the notion of
$\mathbb{A}^{1}$-ruledness, a variety $X$ being called {\it $\mathbb{A}^{1}$\emph{-}ruled}
if it contains a Zariski dense open subset $U$ of the form
$U\simeq\mathbb{A}^{1}\times Y$ for a suitable quasi-projective variety
$Y$. A landmark result of Miyanishi-Sugie \cite{MS80} asserts that
a smooth affine surface is $\mathbb{A}^{1}$-ruled if and only if
it has negative logarithmic Kodaira dimension \cite{Ii77}.
For such surfaces, the projection $\mathrm{pr}_{Y}:U\simeq\mathbb{A}^{1}\times Y\rightarrow Y$
always extends to a fibration $p:X\rightarrow Z$ with general fibers
isomorphic to the affine line $\mathbb{A}^{1}$ over an open subset of the smooth projective
model of $Y$, providing the affine counterpart of the fact that a
smooth birationaly ruled projective surface has the structure of a
fibration with general fibers isomorphic to $\mathbb{P}^{1}$ over
a smooth projective curve. This result, together with the description
of the geometry of degenerate fibers of these fibrations, has been
one of the cornerstones of the structure theory of smooth affine surfaces
developed during the past decades. But in contrast, the foundations
for a systematic study of $\mathbb{A}^{1}$-ruled affine threefolds
have been only laid recently in \cite{GMM12}. On the other hand,
from the point of view of logarithmic Kodaira dimension, the appropriate
counterpart of the notion of uniruledness for a non necessarily complete
variety $X$ is to require that $X$ is generically covered by images
of the affine line $\mathbb{A}^{1}$, in the sense that the set of
points $x\in X$ with the property that there exists a non constant
morphism $f=f_{x}:\mathbb{A}^{1}\rightarrow X$ such that $x\in f_{x}(\mathbb{A}^{1})$
is dense in $X$ with respect to the Zariski topology. Such varieties are called {\it {$\mathbb{A}^{1}$}-uniruled},
or {\it log-uniruled} after Keel and McKernan \cite{KM99}, and
can be equivalently characterized by the property that they admit
an open embedding into a complete variety $\overline{X}$ which is covered
by proper rational curves meeting the boundary $\overline{X}\setminus X$
in at most one point. In particular, a smooth $\mathbb{A}^{1}$-uniruled
quasi-projective variety $X$ has negative logarithmic Kodaira dimension.
It is conjectured that the converse holds true in any dimension, but
so far, the conjecture has been only established in the case of surfaces
by Keel and McKernan \cite{KM99}. 

It follows in particular from these results that for smooth affine
surfaces the notions of $\mathbb{A}^{1}$-ruledness and $\mathbb{A}^{1}$-uniruledness
coincide. Pursuing further the analogy with the classical projective
notions, it seems then natural to expect that these do
no longer coincide for higher dimensional affine varieties. Our main
result confirms that this is indeed the case. Namely, we establish
the following:
\begin{thm*}
For every $n\geq3$, the complement of a smooth hypersurface $Q_{n}$
of degree $n$ in $\mathbb{P}^{n}$ is $\mathbb{A}^{1}$-uniruled
but not $\mathbb{A}^{1}$-ruled.
\end{thm*}
The anti-ampleness of the divisor $K_{\mathbb{P}^{n}}+Q_n$ enables
to easily deduce the $\mathbb{A}^{1}$-uniruledness of affine varieties
of the form $\mathbb{P}^{n}\setminus Q_{n}$ from the general log-deformation
theory for rational curves developed by Keel and McKernan \cite{KM99}.
The failure of $\mathbb{A}^{1}$-ruledness is then obtained in a more
indirect fashion. Indeed, it turns out that for the varieties under
consideration, $\mathbb{A}^{1}$-ruledness is equivalent to the stronger
property that they admit a non trivial action of the additive group
$\mathbb{G}_{a}$. We then exploit two deep results of projective
geometry, namely the non rationality of the smooth cubic threefold
in $\mathbb{P}^{4}$ in the case $n=3$ and the birational super-rigidity
of smooth hypersurfaces $Q_{n}\subset\mathbb{P}^{n}$ if $n\geq4$,
to exclude the existence of such non trivial actions. 

In the last section, we consider more closely the case of complements
of smooth cubic surfaces in $\mathbb{P}^{3}$ which provides a good
illustration of the subtle but crucial difference between the two
notions of $\mathbb{A}^{1}$-ruledness and $\mathbb{A}^{1}$-uniruledness in higher dimension.
We show in particular that even though such complements are not $\mathbb{A}^{1}$-ruled
they admit natural fibrations by $\mathbb{A}^{1}$-ruled affine surfaces.
We study automorphisms of such fibrations in relation with the problem
of deciding whether every automorphism of the complement of a smooth
cubic surface is induced by a linear transformation of the ambient
space $\mathbb{P}^{3}$. 

\section{Recollection on affine ruled and affine uniruled varieties}

\subsection{Affine ruledness and algebraic $\mathbb{G}_{a}$-actions }

Here we review general properties of affine ruled varieties, with
a particular focus on the interplay between $\mathbb{A}^{1}$-ruledness
of an affine variety and the existence of non trivial algebraic actions
of the additive group $\mathbb{G}_{a}$ on it. We refer the reader 
to \cite{Freu} for basic properties of the correspondence between
such actions on affine varieties and their algebraic counterpart, the so-called locally nilpotent
derivations of their coordinate rings. 
\begin{defn}
A quasi-projective variety $X$ is called \emph{affine ruled} or $\mathbb{A}^{1}$\emph{-ruled}
if it contains a Zariski dense open subset $U$ of the form
$U\simeq\mathbb{A}^{1}\times Y$ for a suitable quasi-projective variety
$Y$. An open subset 
of this type is called an \emph{$\mathbb{A}^{1}$-cylindrical
open subset} of $X$. 
\end{defn}
\begin{parn} An $\mathbb{A}^{1}$-ruled variety $X$ is in particular
ruled in the usual sense, the inclusion $\mathbb{A}^{1}\times Y\hookrightarrow X$
of an $\mathbb{A}^{1}$-cylindrical open subset of $X$ inducing a dominant
birational map $\mathbb{P}^{1}\times Y\dashrightarrow X$. The converse
is not true since for instance the product of the punctured affine
line $\mathbb{A}_{*}^{1}=\mathbb{A}^{1}\setminus\{0\}$ with a smooth
projective curve of positive genus is ruled but does not contain any
$\mathbb{A}^{1}$-cylinder. In particular, in contrast with ruledness,
the property of being $\mathbb{A}^{1}$-ruled is not invariant under
birational equivalence. Note similarly that the existence of an $\mathbb{A}^{1}$-cylindrical
open subset of $X$ is a stronger requirement than that of a dominant
birational morphism $\mathbb{A}^{1}\times Y\rightarrow X$ for a suitable
variety $Y$: for instance, the affine cone $X\subseteq \mathbb{A}^{3}$
over a projective plane curve $C\subseteq \mathbb{P}^{2}$ of positive genus
does not contain $\mathbb{A}^{1}$-cylinders but on the other hand,
blowing-up the vertex of the cone $X$ yields a smooth quasi-projective
variety $\sigma:\tilde{X}\rightarrow X$ with the structure of a locally
trivial $\mathbb{A}^{1}$-bundle $\rho:\tilde{X}\rightarrow C$ hence,
restricting to a subset $Y$ of $C$ over which $\rho$ is trivial,
a dominant birational morphism $\sigma\mid_{{\rho^{-1}}(Y)}:\rho^{-1}(Y)\simeq\mathbb{A}^{1}\times Y\rightarrow X$. 

\end{parn}

\begin{parn} \label{par:Ga-cylinder} Typical examples of $\mathbb{A}^{1}$-ruled
varieties are normal affine varieties $X=\mathrm{Spec}(A)$ admitting
a nontrivial algebraic action of the additive group $\mathbb{G}_{a}$.
Indeed, if $X$ is equipped with such an action induced by a locally
nilpotent $\mathbb{C}$-derivation $\partial$ of $A$, then for every
local slice $f\in\mathrm{Ker}\partial^{2}\setminus\mathrm{Ker}\partial$,
the principal open subset $X_{\partial f}=\mathrm{Spec}(A_{\partial f})$, where $A_{\partial f}$ denotes the localization $A[(\partial f)^{-1}]$ of $A$, is $\mathbb{G}_{a}$-invariant and the morphism 
$\mathbb{G}_{a}\times(V(f)\cap X_{\partial f})\rightarrow X_{\partial f}$
induced by the $\mathbb{G}_{a}$-action on $X$ is an isomorphism.
In particular, $X_{\partial_{f}}$ is a principal $\mathbb{A}^{1}$-cylindrical
open subset of $X$. 
Note on the contrary that the existence of an $\mathbb{A}^{1}$-cylindrical open subset $U\simeq\mathbb{A}^{1}\times Y$ of an affine variety $X=\mathrm{Spec}(A)$
is in general not enough to guarantee that $X$ can be equipped with
a $\mathbb{G}_{a}$-action whose general orbit{\bf s} coincide with the general
fibers of the projection $\mathrm{pr}_{Y}:U\simeq\mathbb{A}^{1}\times Y\rightarrow Y$.
For instance, the general fibers of the projection $\mathrm{pr}_{1}:X=\mathbb{P}^{1}\times\mathbb{P}^{1}\setminus\Delta\rightarrow\mathbb{P}^{1}$,
where $\Delta\subset\mathbb{P}^{1}\times\mathbb{P}^{1}$ denotes the
diagonal, cannot coincides with general orbits of an algebraic $\mathbb{G}_{a}$-action
on $X$. Indeed, otherwise every invariant function on $X$ would
descend to a regular function on $\mathbb{P}^{1}$ and hence would
be constant, in contradiction with the fact that for an affine variety
$X$, the field $\mathrm{Frac}(\Gamma(X,\mathcal{O}_{X})^{\mathbb{G}_{a}})$
has transcendence degree $\dim X-1$ over $\mathbb{C}$. 

In fact, it is classically known that the general fibers of an $\mathbb{A}^{1}$-cylindrical
affine open subset $U\simeq\mathbb{A}^{1}\times Y$ of a normal affine
variety $X=\mathrm{Spec}(A)$ coincide with the general orbits of
an algebraic $\mathbb{G}_{a}$-action on $X$ if and only if $U$
is a principal affine open subset of $X$ (see e.g.  \cite[Proposition 3.1.5]{KPZ}). Let us briefly recall the
argument for the convenience of the reader: the existence of a principal
$\mathbb{A}^{1}$-cylinder in $X$ is equivalent to the existence
of an element $a\in A\setminus\left\{ 0\right\} $ and an isomorphism
of $\mathbb{C}$-algebra $\varphi:A_{a}\stackrel{\sim}{\rightarrow}B[t]$, 
where $t$ is transcendental over $\mathrm{Frac}(B)$. The derivation
$\partial_{0}=\varphi^{*}\frac{\partial}{\partial t}$ of $A_{a}$
is then locally nilpotent and since $A$ is a finitely generated algebra,
there exists $n\geq0$ such that $a^{n}\partial_{0}$ induces a $\mathbb{C}$-derivation
$\partial$ of $A$. Noting that being invertible in $A_{a}$, $a$
is necessarily contained in the kernel of $\partial_{0}$, we conclude
that $\partial$ is a locally nilpotent derivation of $A$, defining
a $\mathbb{G}_{a}$-action on $X$ whose restriction to the principal
invariant open subset $U={\rm Spec}  (A_{a}) \simeq\mathrm{Spec}(B)\times\mathbb{A}^{1}$
is by construction equivariantly isomorphic to that by translations
on the second factor. 

It follows in particular that for a normal affine variety $X$ whose
divisor class group $\mathrm{Cl}(X)$ is torsion, the existence of $\mathbb{A}^{1}$-cylindrical open
subsets is essentially equivalent to that of nontrivial $\mathbb{G}_{a}$-actions.
More precisely, we have the following criterion: 

\end{parn}
\begin{prop}
\label{prop:Ga-Crit} Let $X=\mathrm{Spec}(A)$ be a normal affine
variety such that $\mathrm{Cl}(X)\otimes_{\mathbb{Z}}\mathbb{Q}=0$.
Then for every $\mathbb{A}^{1}$-cylindrical open subset $U\simeq\mathbb{A}^{1}\times Y$
of $X$ there exists an action of $\mathbb{G}_{a}$ on $X$ whose
general orbits coincide with the general fibers of the projection
$\mathrm{pr}_{Y}:U\rightarrow Y$. In particular, $X$ admits an $\mathbb{A}^{1}$-cylinder
open subset if and only if it admits a nontrivial $\mathbb{G}_{a}$-action. \end{prop}
\begin{proof}
By replacing $Y$ by an affine open subset of it, we may assume that
$U$ is affine whence that its complement $D=X\setminus U$
has pure codimension one in $X$ as $X$ itself is affine. The hypothesis
guarantees precisely that $D$ is the support of a principal divisor
and hence the assertions follow immediately from the discussion above. 
\end{proof}
\subsection{Basic facts on log uniruled quasi-projective varieties}
\begin{defn}
A quasi-projective variety $X$ is called \emph{$\mathbb{A}^{1}$-uniruled},
or (\emph{properly})\emph{ log-uniruled} after \cite{KM99} if it
contains a Zariski dense open subset $U$ with the property
that for every $x\in U$, there exists a non constant morphism $f_{x}:\mathbb{A}^{1}\rightarrow X$
such that $x\in f_{x}(\mathbb{A}^{1})$. 
\end{defn}
\begin{parn} Equivalently, $X$ is $\mathbb{A}^{1}$-uniruled if
through a general point there exists a maximally affine rational curve,
that is, an affine rational curve whose normalization is isomorphic to
the affine line $\mathbb{A}^{1}$. An $\mathbb{A}^{1}$-uniruled variety is
in particular uniruled in the usual sense, i.e. there exists a dominant,
generically finite rational map $\mathbb{P}^{1}\times Y\dashrightarrow X$
for a suitable quasi-projective variety $Y$. More precisely, letting $(V,D)$
be a pair consisting of a projective model $V$ of $X$ and a, possibly
empty, boundary divisor $D$ such that $V$ is smooth along
$D$, it follows from \cite[5.1]{KM99} that there exists a closed
sub-scheme $Y$ of $\mathrm{Mor}(\mathbb{P}^{1},V)$ consisting
of morphisms $f:\mathbb{P}^{1}\rightarrow V$ with the property
that $f^{-1}(f(\mathbb{P}^{1})\cap D)$ consists of at most one point
and on which the restriction of the canonical evaluation morphism
$\mathrm{ev}:\mathbb{P}^{1}\times\mathrm{Mor}(\mathbb{P}^{1},V)\rightarrow V$
induces a dominant morphism $\mathbb{P}^{1}\times Y\rightarrow V$. 

\end{parn}

\begin{parn} It is not difficult to check that a smooth $\mathbb{A}^{1}$-uniruled
quasi-projective variety $X$ has logarithmic Kodaira dimension $\overline{\kappa}(X)=-\infty$
(see e.g. \cite[5.11]{KM99}). The converse holds true in dimension
$\leq2$ thanks to the successive works of Miyanishi-Sugie \cite{MS80}
and Keel-McKernan \cite{KM99} and in fact, for a smooth affine surface
$X$, the three conditions $\overline{\kappa}(X)=-\infty$, $X$ is
$\mathbb{A}^{1}$-uniruled and $X$ is $\mathbb{A}^{1}$-ruled turn
out to be equivalent to each other. Furthermore, the surface $X$
has then the stronger property that every of its point belongs to
a maximally affine rational curve. This follows from the fact that
the projection $\mathrm{pr}_{Y}:U\simeq\mathbb{A}^{1}\times Y\rightarrow Y$
from an $\mathbb{A}^{1}$-cylindrical open subset $U$ extends to a fibration $p:X\rightarrow Z$
over an open subset of a smooth projective model of $Y$, with general
fibers isomorphic to $\mathbb{A}^{1}$ and whose degenerate fibers
consist of disjoint unions of affine lines (see e.g. \cite{Mi81}).  
Meanwhile, much less is known in higher dimension where one lacks in particular
a good logarithmic analogue of Mori's Bend-and-Break techniques. Nevertheless,
the following simple criterion due to Keel and McKernan \cite[Corollary 5.4]{KM99}
enables to easily confirm $\mathbb{A}^{1}$-ruledness of certain smooth
affine varieties and will be enough for our purpose: 

\end{parn}
\begin{thm}
\label{thm:Log-Uni-Crit} Let $X$ be a smooth affine variety and
let $X\hookrightarrow(V,D)$ be a projective completion where $V$
is a smooth projective variety and $D=V\setminus X$ a reduced divisor
on $V$. If $-(K_{V}+D)$ is ample then $X$ is $\mathbb{A}^{1}$-uniruled. \end{thm}
\begin{example}
\label{Ex:Log-Uni-Ex} The above criterion guarantees for instance
that for every $n\geq1$, the complement $X$ of a hypersurface $D\subseteq \mathbb{P}^{n}$
of degree $d\leq n$ is $\mathbb{A}^{1}$-uniruled. In dimension $2$,
we recover in particular the fact that the complement of smooth conic
$Q_{2}$ in $\mathbb{P}^{2}$ is $\mathbb{A}^{1}$-uniruled, actually
$\mathbb{A}^{1}$-ruled say for instance by the restriction to $\mathbb{P}^{2}\setminus Q_{2}$
of the rational pencil generated by $Q_{2}$ and two times its tangent line at a given point.
On the other hand, even though the criterion does not require $D$
to be an SNC divisor, it does not enable to deduce the fact that the
complement of a cuspidal cubic $D$ is $\mathbb{A}^{1}$-uniruled,
in fact again $\mathbb{A}^{1}$-ruled, namely via the restriction
of the rational pencil on $\mathbb{P}^{2}$ generated by $D$ and three
times its tangent line at the singular point. 
\end{example}
\section{Complement of smooth hypersurfaces of degree $n$ in $\mathbb{P}^{n}$,
$n\geq3$ }

According to Theorem \ref{thm:Log-Uni-Crit}, the complement $X$
of a hypersurface $Q_{n}$ of degree $n$ in $\mathbb{P}^{n}$,
$n\geq2$, is a smooth $\mathbb{A}^{1}$-uniruled affine variety.
If $n=2$, then $Q_{2}$ is {\bf a} smooth conic whose complement is in fact
$\mathbb{A}^{1}$-ruled (see Example \ref{Ex:Log-Uni-Ex}). Here we
show in contrast that for every $n\geq3$, the complement of a smooth
hypersurface $Q_{n}\subseteq \mathbb{P}^{n}$ is not $\mathbb{A}^{1}$-ruled. 
It is worthwhile to note that the strategy to establish this fact in case of $n=3$ is quite different from that of the case of $n\ge 4$. 

\subsection{The case $n\geq4$ }

Here we will exploit the birational super-rigidity of smooth hypersurfaces
$Q_{n}\subset\mathbb{P}^{n}$ of degree $n\geq 4$ to deduce that the corresponding
affine varieties $X=\mathbb{P}^{n}\setminus Q_{n}$ are not $\mathbb{A}^{1}$-ruled.
Let us first briefly recall the notion of birational super-rigidity
for a class of variety which is enough for our needs (see e.g. \cite{Pu98}
for the general definition). 
\begin{defn}
\label{Def:Bir-SupRig} A smooth Fano variety $V$ with $\mathrm{Pic}(V)\simeq\mathbb{Z}$
is called \emph{birationaly} \emph{super-rigid} if the following two
conditions hold:

a) The variety $V$ is not birational to a fibration $V'\rightarrow S$ onto a variety of ${\rm dim} S >0$
whose general fibers are smooth varieties with Kodaira dimension $-\infty$, 

b) Any birational map from $V$ to another smooth Fano variety $V'$
with $\mathrm{Pic}(V')\simeq\mathbb{Z}$ is a biregular isomorphism. 
\end{defn}
\noindent The essential ingredient we will use is the following result,
first established by Pukhlikov \cite{Pu98} under suitable generality
assumptions and recently extended to arbitrary smooth hypersurfaces
by De Fernex \cite{dF06}: 
\begin{thm}\label{thmdF}
Let $n\geq4$ and let $Q_{n}\subseteq \mathbb{P}^{n}$ be a smooth hypersurface
of degree $n$. Then $Q_n$ is birationaly super-rigid. \end{thm}
\begin{rem}
\label{Rk:BirSupRig-Auto} A noteworthy consequence of the above result
is that every birational map $Q_{n}\dashrightarrow Q_{n}'\subset\mathbb{P}^{n}$
between smooth hypersurfaces of degree $n$ is in fact a biregular isomorphism
which has the additional property to be induced by the restriction
of a linear transformation of the ambient space $\mathbb{P}^{n}$.
This follows from the fact that anti-canonical divisors on $Q_{n}$
coincide with hyperplane sections of $Q_{n}$. In particular, the
group $\mathrm{Bir}(Q_{n})$ of birational automorphisms of $Q_{n}$
coincides with that of projective automorphisms $\mathrm{Lin}(Q_{n})$
and the latter is a finite group. 
\end{rem}
\noindent Now we are ready to prove the following:
\begin{prop}
\label{prop:Non-ruled-n4} The complement of every smooth hypersurface
$Q_{n}\subseteq \mathbb{P}^{n}$ of degree $n\geq4$ is $\mathbb{A}^{1}$-uniruled
but not $\mathbb{A}^{1}$-ruled. \end{prop}
\begin{proof}
The $\mathbb{A}^{1}$-uniruledness of $X=\mathbb{P}^{n}\setminus Q_{n}$
is an immediate consequence of Theorem \ref{thm:Log-Uni-Crit}. Since $\mathrm{Cl}(X)=\mathbb{Z}/n\mathbb{Z}$,
it follows from Proposition \ref{prop:Ga-Crit} that the $\mathbb{A}^{1}$-ruledness
of $X$ is equivalent to the existence of a non trivial algebraic
$\mathbb{G}_{a}$-action on it. Every given $\mathbb{G}_{a}$-action on
$X$ induces a one parameter family of automorphisms $\{\varphi_{t}\}_{t\in\mathbb{G}_{a}}$
of $X$ that we interpret through the open inclusion $X=\mathbb{P}^{n}\setminus Q_{n}\hookrightarrow\mathbb{P}^{n}$
as birational automorphisms $\Phi_{t}:\mathbb{P}^{n}\dashrightarrow\mathbb{P}^{n}$,
$t\in\mathbb{G}_{a}$. Since $Q_{n}$ is birationally super-rigid by virtue of Theorem \ref{thmdF} above,
it follows that every $\Phi_{t}$ is in fact a biregular automorphism.
Indeed, otherwise, noting that $\Phi_{t}$ cannot be an isomorphism in
codimension one (see e.g. \cite{Cor95}) and letting $\mathbb{P}^{n}\stackrel{p}{\leftarrow}V\stackrel{q}{\rightarrow}\mathbb{P}^{n}$
be a resolution of it, where $p$ consists of successive blow-ups
of smooth centers, it would follow that the proper transform $q_{*}^{-1}Q_{n}$
of $Q_{n}$ in $V$ is an exceptional divisor of $p$ whence is birationally ruled
in contradiction with condition a) in Definition \ref{Def:Bir-SupRig}.
Therefore the $\mathbb{G}_{a}$-action on $X$ extends to a linear
one on $\mathbb{P}^{n}$ which leaves $X$ whence $Q_{n}$ invariant.
Since the automorphism group $\mathrm{Aut}(Q_{n})$ is finite (see
Remark \ref{Rk:BirSupRig-Auto} above), the induced $\mathbb{G}_{a}$-action
on $Q_{n}$ is the trivial one and so, being linear, the $\mathbb{G}_{a}$-action
on $\mathbb{P}^{n}$ extending that on $X$ is trivial on the linear
span of $Q_{n}$. But $Q_{n}$ is obviously not contained in any hyperplane
of $\mathbb{P}^{n}$ and hence the initial $\mathbb{G}_{a}$-action
on $X$ is necessarily trivial, which completes the proof. 
\end{proof}

\subsection{The case $n=3$: complements of smooth cubic surfaces in $\mathbb{P}^{3}$ }

Since a smooth cubic surface $Q_{3}$ in $\mathbb{P}^{3}$ is rational, 
the same argument as in the previous section depending on birational
super-rigidity is no longer applicable to deduce the non $\mathbb{A}^{1}$-ruledness
of its complement $X=\mathbb{P}^{3}\setminus Q_{3}$. Instead, we
will derive it from the unirationality but non rationality of smooth
cubic threefolds in $\mathbb{P}^{4}$ together with a suitable finite
\'etale covering trick. 

The divisor class group of the complement $X$
of a smooth cubic surface $Q_{3}\subseteq \mathbb{P}^{3}$ being isomorphic
to $\mathbb{Z}/3\mathbb{Z}$ generated by the class of a hyperplane
section, it follows again from Proposition \ref{prop:Ga-Crit} that
$X$ is $\mathbb{A}^{1}$-ruled if and only if it admits a non trivial
algebraic $\mathbb{G}_{a}$-action. On the other hand, since algebraic
$\mathbb{G}_{a}$-actions lift under finite \'etale covers (see e.g.
\cite{Sei66}), to establish the non $\mathbb{A}^{1}$-ruledness of
$X$ it is enough to exhibit a finite \'etale cover $\pi:\tilde{X}\rightarrow X$
of $X$ with the property that $\tilde{X}$ does not admit any non
trivial algebraic $\mathbb{G}_{a}$-action. 
\begin{prop}\label{prop1}
Let $X=\mathbb{P}^{3}\setminus Q_{3}$ be the complement of a smooth
cubic surface and let $\pi:\tilde{X}\rightarrow X$ be the canonical
\'etale Galois cover of degree three associated with the canonical
sheaf $\omega_{X}\simeq\mathcal{O}_{\mathbb{P}^{1}}(-1)\mid_{X}$
of $X$. Then $\tilde{X}$ is a smooth affine threefold which does
not admit any non trivial algebraic $\mathbb{G}_{a}$-action. Consequently,
$X$ is not $\mathbb{A}^{1}$-ruled. \end{prop}
\begin{proof}
Letting $Q_{3}\subseteq \mathbb{P}^{3}=\mathrm{Proj}(\mathbb{C}[x,y,z,u])$
be defined as the zero locus of a homogeneous polynomial $F(x,y,z,u)\in\mathbb{C}[x,y,z,u]$
of degree three, it is straightforward to check that the canonical triple \'etale covering $\tilde{X}$ of $X$
is isomorphic to the affine variety $\mathbb{A}^{4}=\mathrm{Spec}(\mathbb{C}[x,y,z,u])$
defined by the equation $F(x,y,z,u)=1$, and that the morphism $\pi:\tilde{X}\rightarrow X$ 
coincides with the restriction of the natural morphism
$\mathbb{A}^{4}\setminus\left\{ 0\right\} \rightarrow\mathbb{P}^{3}$,
$(x,y,z,u)\mapsto[x:y:z:u]$. Note that by construction, $\tilde{X}$ is an open subset of the smooth cubic threefold $V\subseteq \mathrm{Proj}(\mathbb{C}[x,y,z,u,v])$ with equation $F(x,y,z,u)-v^{3}=0$, hence is unirational but not rational by virtue of a famous result of Clemens-Griffiths \cite{CM72}. Now suppose that
there exists a non trivial algebraic $\mathbb{G}_{a}$-action on $\tilde{X}$.
Then by virtue of \cite{Zar54}, the algebra of invariants $\Gamma(\tilde{X},\mathcal{O}_{\tilde{X}})^{\mathbb{G}_{a}}$
is a finitely generated integrally closed domain of transcendence
degree two over $\mathbb{C}$. Letting $q:\tilde{X}\rightarrow\tilde{Z}=\mathrm{Spec}(\Gamma(\tilde{X},\mathcal{O}_{\tilde{X}})^{\mathbb{G}_{a}})$ be the corresponding quotient morphism, the unirationality of $\tilde{X}$
implies that of $\tilde{Z}$ whence its rationality since these two
notions coincide in dimension two. But then the existence of a principal
affine open subset $\tilde{Z}_{a}$ of $\tilde{Z}$ such that $q^{-1}(\tilde{Z}_{a})\simeq\mathbb{A}^{1}\times\tilde{Z}_{a}$
(see $\S$ \ref{par:Ga-cylinder} above) would imply in turn that $\tilde{X}$
itself is rational, a contradiction.\end{proof}

\begin{rem}
Complements of singular cubic surfaces $Q_{3}\subseteq\mathbb{P}^{3}$
may turn out to be $\mathbb{A}^{1}$-ruled. For instance, it is shown
in \cite{KPZ11} that if $Q_{3}$ has Du Val singularities worse than
$A_{2}$ then $\mathbb{P}^{3}\setminus Q_{3}$ admits a non trivial
algebraic $\mathbb{G}_{a}$-action. It may even happen that $\mathbb{P}^{3}\setminus Q_{3}$
contains an $\mathbb{A}^{2}$-cylinder, i.e., an open subset of the
form $\mathbb{A}^{2}\times Y$ for a smooth curve $Y$. This holds
for instance for the complement of the cubic surface $Q\subseteq \mathbb{P}^{3}=\mathrm{Proj}(\mathbb{C}[x,y,z,u])$
defined as the vanishing locus of the homogeneous polynomial $F(x,y,z,u)=yu^{2}+z(xz+y^{2})$
which has a unique isolated singularity $\left[1:0:0:0\right]$ of
type $D_{5}$. Indeed, noting that $f(x,y,z)=F(x,y,z,1)$ is a component
of the Nagata automorphism \cite{Na78}, it follows that the general
fibers of the morphism $\rho:\mathbb{P}^{3}\setminus Q\rightarrow\mathbb{A}^{1}$
induced by the restriction of the rational pencil on $\mathbb{P}^{3}$
generated by $Q$ and three times the hyperplane $\{u=0\}$
are isomorphic to $\mathbb{A}^{2}$ hence, by virtue of \cite{KZ01},
that there exists an open subset $Y\subseteq \mathbb{A}^{1}$ such that
$\rho^{-1}(Y)\simeq\mathbb{A}^{2}\times Y$. A similar construction
holds for all normal cubic surfaces listed in \cite{Oh01} which arise
as closures in $\mathbb{P}^{3}$ of zero sets of components of automorphisms
of $\mathbb{A}^{3}$. 
\end{rem}
\section{Rigid rational fibrations on complements of smooth cubic surfaces
in $\mathbb{P}^{3}$}

In contrast with the higher dimensional case where a similar argument
as the one used in the proof of Proposition \ref{prop:Non-ruled-n4}
shows that for every smooth hypersurface $Q_{n}$ of the degree $n$
in $\mathbb{P}^{n}$, $n\geq4$, the group $\mathrm{Aut}(\mathbb{P}^{n}\setminus Q_{n})$ 
is embeded as a subgroup of $\mathrm{Aut}(Q_{n})$ hence consists of a
finite group of linear transformations, it is not known whether every
automorphism of the complement of a smooth cubic surface $Q=Q_3$ in $\mathbb{P}^{3}$
is induced by a linear transformation of $\mathbb{P}^{3}$. 

Even though ${\mathbb P}^3\backslash Q$ is not $\mathbb{A}^{1}$-ruled (cf. Proposition \ref{prop1}), it turns out that it is fibered in a natural way by $\mathbb{A}^{1}$-ruled affine surfaces whose general closures in
$\mathbb{P}^{3}$ are smooth cubic surfaces. Since $\mathbb{A}^{1}$-ruled
affine surfaces are usually good candidates for having non trivial
automorphisms, one can expect that these fibrations have interesting
automorphisms, in particular automorphisms induced by strictly birational
transformations of the ambient space $\mathbb{P}^{3}$. 

In this section, we first review the construction of these fibrations
$\rho:\mathbb{P}^{3}\setminus Q\rightarrow\mathbb{A}^{1}$ by ${\mathbb A}^1$-ruled affine surfaces. We describe
the automorphism groups of their general fibers and we check in particular
that for a general smooth cubic hypersurface $Q\subseteq {\mathbb P}^3$, these fibers do indeed carry interesting automorphisms
induced by strictly birational Geiser involutions of their projective
closures. We show in contrast that every automorphism of the full
fibration comes as the restriction of a linear transformation
of $\mathbb{P}^{3}$. 

\subsection{Special rational pencils on the complement of a smooth cubic surface}\label{3.1}

Given a smooth cubic surface $Q\subseteq \mathbb{P}^{3}$ and a line
$L$ on it, the restriction to $Q$ of the rational pencil $\mathcal{H}_{L}=\left|\mathcal{O}_{\mathbb{P}^{3}}(1)\otimes\mathcal{I}_{L}\right|$ on $\mathbb{P}^{3}$ generated by planes containing $L$ can be decomposed
as $\mathcal{H}_{L}\mid_{Q}=\mathcal{L}+L$ where $\mathcal{L}$ is
a base point free pencil defining a conic bundle $\Phi_{\mathcal{L}}:Q\rightarrow\mathbb{P}^{1}$
with five degenerate fibers, each consisting of the union of two $\left(-1\right)$-curves
intersecting transversally. The restriction $\Phi_{\mathcal{L}}\mid_{L}:L\rightarrow\mathbb{P}^{1}$
is a double cover and for every branch point $x\in\mathbb{P}^{1}$
of $\Phi_{\mathcal{L}}\mid_{L}$, the intersection of $Q$ with the
corresponding hyperplane $H_{x}\in\mathcal{H}_L$ consists either of
a smooth conic tangent to $L$ or of two distinct lines intersecting
$L$ in a same point, which is called an Eckardt point of $Q$. Letting
say $H_{0}(L)$ and $H_{\infty}(L)$ be the planes in $\mathcal{H}_{L}$
such that $\Phi_{\mathcal{L}}\mid_{L}$ is ramified over the points
$\Phi_{\mathcal{L}}(H_{t}(L)\mid_{Q}-L)$, $t=0,\infty$, the divisors
$Q$ and $3H_{t}(L)$ generate a pencil of cubic surfaces $\overline{\rho}_{t}(L):\mathbb{P}^{3}\dashrightarrow\mathbb{P}^{1}$
with a unique multiple member $3H_{t}(L)$ and whose general members
are smooth cubic surfaces. 

\begin{parn} \label{par:Log-res} From now on we consider a pencil
$\overline{\rho}=\overline{\rho}_{t}(L)$ as above associated to a fixed 
line $L\subseteq Q$ and a fixed distinguished plane $H\in\mathcal{H}_{L}$
intersecting $Q$ either along the union $L\cup C$ of a line and
conic intersecting each other in a single point $p$ or along the
union $L\cup L_{1}\cup L_{2}$ of three lines meeting at an Eckardt
point $p$ of $Q$. We denote by $\rho:\mathbb{P}^{3}\setminus Q\rightarrow\mathbb{A}^{1}$
the morphism induced by the restriction of $\overline{\rho}$ to the
complement of $Q$. By construction, the closure in $\mathbb{P}^{3}$
of a general fiber $S$ of $\rho$ is a smooth cubic surface $V$
such that the reduced divisor $D=H\cap V$ has either the form $D=L+C$
or $D=L+L_{1}+L_{2}$. In each case we denote by $\alpha:W\rightarrow V$
the blow-up of $p$, with exceptional divisor $E$. If $p$ is an
Eckardt point, then we let $\tilde{V}=W$ and we denote by $\tilde{D}\subseteq \tilde{V}$
the reduced divisor $\alpha^{-1}(D)_{\mathrm{red}}=L+L_{1}+L_{2}+E$.
Otherwise, if $p$ is not an Eckardt point, then we let $\tilde{V}$
be the variety obtained from $W$ by blowing-up further the intersection
point of $E$ and of the proper transform of $C$, say with exceptional
divisor $E_{2}$, and we let $\tilde{D}\subseteq \tilde{V}$ be the reduced
divisor $\tilde{D}=L+C+E+E_{2}$. By construction, $\tilde{D}$ is
an SNC divisor and the induced birational morphism $\sigma:(\tilde{V},\tilde{D})\rightarrow(V,D)$
provides a minimal log-resolution of the pair $\left(V,D\right)$. 

\end{parn}

\noindent The following lemma summarizes basic properties of the
general fibers of fibrations of the form $\rho_{t}(L):\mathbb{P}^{3}\setminus Q\rightarrow\mathbb{A}^{1}$,
$t=0,\infty$. 
\begin{lem}\label{basic} 
For a general fiber $S=V\setminus D$ of
$\rho_{t}(L):\mathbb{P}^{3}\setminus Q\rightarrow\mathbb{A}^{1}$,
the following holds:

a) $S$ is a smooth affine surface with a trivial canonical sheaf and
logarithmic Kodaira dimension $\overline{\kappa}(S)=-\infty$, 

b) If $D=L+C$ (resp. $D=L+L_{1}+L_{2}$) then the Picard group of
$S$ is isomorphic to $\mathbb{Z}^{5}$ \emph{(}resp. $\mathbb{Z}^{4}$\emph{)}. \end{lem}
\begin{proof}
The affineness of $S$ is clear as $D$ is a hyperplane section of
$V$. Since $V$ is a smooth cubic surface, it follows from adjunction
formula that $\omega_{V}\simeq\omega_{\mathbb{P}^{3}}\otimes\mathcal{O}_{\mathbb{P}^{3}}(3)\mid_{V}\simeq\mathcal{O}_{\mathbb{P}^{3}}(-1)\mid_{V}\simeq\mathcal{O}_{V}(-D)$.
This implies in turn the triviality of the canonical sheaf of $S$ as $\omega_{S}\simeq\omega_{V}\mid_{S}\simeq\mathcal{O}_{V}(-D)\mid_{S}\simeq\mathcal{O}_{S}$.
We may identify $S$ with $\tilde{V}\setminus\tilde{D}$ via the birational
morphism $\sigma:(\tilde{V},\tilde{D})\rightarrow(V,D)$ constructed
above. Since $\tilde{D}$ is an SNC divisor by construction, we have
$\overline{\kappa}(S)=\kappa(\tilde{V},K_{\tilde{V}}+\tilde{D})$.
Using the fact that $-D$ is a canonical divisor on $V$, we deduce
from the logarithmic ramification formula for $\sigma$ that 
\[
\begin{cases}
K_{\tilde{V}}+\tilde{D}=-E & \textrm{if }D=L+L_{1}+L_{2}\\
K_{\tilde{V}}+\tilde{D}=-E_{2} & \textrm{if }D=L+C.
\end{cases}
\]
Since $E$ and $E_{2}$ are exceptional divisors, one has $H^{0}(\tilde{V},m(K_{\tilde{V}}+\tilde{D}))=0$
for every $m>0$ and hence $\overline{\kappa}(S)=-\infty$. 

To determine the Picard group of $S$, we will exploit the conic bundle
structure $\Phi_{\mathcal{L}}:V\rightarrow\mathbb{P}^{1}$ described
in the beginning of \ref{3.1}. By contracting a suitable irreducible component of each of
the five degenerate fibers of $\Phi_{\mathcal{L}}:V\rightarrow\mathbb{P}^{1}$,
we obtain a birational morphism $\tau:V\rightarrow\mathbb{P}^{1}\times\mathbb{P}^{1}$
fitting into a commutative diagram 
\begin{eqnarray*}
V & \stackrel{\tau}{\rightarrow} & \mathbb{P}^{1}\times\mathbb{P}^{1}\\
\Phi_{\mathcal{L}}\downarrow &  & \downarrow\mathrm{pr}_{1}\\
\mathbb{P}^{1} & = & \mathbb{P}^{1}
\end{eqnarray*}
and such that the proper transform of $L$ is a smooth $2$-section
of $\mathrm{pr}_{1}$ with self-intersection $4$, whence a section
of the second projection $\mathrm{pr}_{2}:\mathbb{P}^{1}\times\mathbb{P}^{1}\rightarrow\mathbb{P}^{1}$.
We can then identify the divisor class group $\mathrm{Cl}(V)\simeq\mathbb{Z}^{7}$
of $V$ with the group generated by the proper transforms of a fiber
of $\mathrm{pr_{1}}$, a fiber of $\mathrm{pr}_{2}$, and the five
exceptional divisors of $\tau$. The proper transform $\tau_{*}(D)$
of $D$ in $\mathbb{P}^{1}\times\mathbb{P}^{1}$ is the union of the
proper transform of $L$ and of a fiber of $\mathrm{pr}_{1}$. Thus
if $D=L+C$ then $L$ and $C$ together with the five exceptional
divisors of $\tau$ generate $\mathrm{Cl}(V)$ and hence $\mathrm{Cl}(S)\simeq\mathrm{Pic}(S)$
is isomorphic to $\mathbb{Z}^{5}$ generated by the classes of the
intersections of these exceptional divisors with $S$. Similarly,
if $D=L+L_{1}+L_{2}$ then $L_{1}$ or $L_{2}$, say $L_{2}$, is
an exceptional divisor of $\tau$ and so $\mathrm{Cl}(V)$ is generated
by $L$, $L_{1}$, $L_{2}$ and the four other exceptional divisors
of $\tau$. This implies in turn that $\mathrm{Cl}(S)\simeq\mathrm{Pic}(S)$
is isomorphic to $\mathbb{Z}^{4}$ generated by the classes of the
intersections of these four exceptional divisors with $S$. 
\end{proof}
\begin{parn} \label{Rk:A1-fib} Since a general fiber $S=V\setminus D$ of $\rho_{t}(L):\mathbb{P}^{3}\setminus Q\rightarrow\mathbb{A}^{1}$ has logarithmic Kodaira dimension $-\infty$ by Lemma \ref{basic},  
it is $\mathbb{A}^{1}$-ruled and hence, by virtue of \cite{Mi81}, it
admits an $\mathbb{A}^{1}$-fibration $q:S\rightarrow C$ over a smooth
curve. Let us briefly explain how to construct such fibrations using
the birational morphism $\tau:V\rightarrow\mathbb{P}^{1}\times\mathbb{P}^{1}$
which contracts an irreducible component of each of the five degenerate
fibers of $\Phi_{\mathcal{L}}:V\rightarrow\mathbb{P}^{1}$ as described
in the proof of the previous Lemma. According to the configuration of $D$, we have the following: 

a) If $D=L+C$ where $C$ is a smooth conic, then each of the points blown-up
by $\tau$ is the intersection point of the proper transform of $L$
with a fiber of $\mathrm{pr}_{2}:\mathbb{P}^{1}\times\mathbb{P}^{1}\rightarrow\mathbb{P}^{1}$.
The proper transforms $F_{1},\ldots,F_{5}\subseteq  V$ of these fibers
are disjoint $\left(-1\right)$-curves which do not intersect $L$.
Let $\mu:V\rightarrow\mathbb{P}^{2}$ be the contraction of $L$ and
$F_{1},\ldots,F_{5}$. Since each $F_{i}$, $i=1,\ldots,5$ intersects
$C$ transversally and $L$ is tangent to $C$, the image of $C$
in $\mathbb{P}^{2}$ is a cuspidal cubic. The rational pencil on $\mathbb{P}^{2}$
generated by the image of $C$ and three times its tangent line $T$ at
its unique singular point lifts to a rational pencil $\overline{q}:V\dashrightarrow\mathbb{P}^{1}$
having the divisors $C+\sum_{i=1}^{5}F_{i}$ and $3T+L$ as singular
members. Its restriction to $S=V\setminus D$ is an $\mathbb{A}^{1}$-fibration
$q:S\rightarrow\mathbb{P}^{1}$ with two degenerate fibers: one is
irreducible of multiplicity three consisting of the intersection of
the proper transform of $T$ with $S$ and the other is reduced, consisting
of the disjoint union of the curves $F_{i}\cap S\simeq\mathbb{A}^{1}$,
$i=1,\ldots,5$. 

b) The construction in the case where $D=L+L_{1}+L_{2}$ is very similar: 
we may suppose that $L_{1}$ is contracted by $\tau$ so that contracting
the same irreducible components $F_{1},\ldots,F_{4}$ of the degenerate
fibers of $\Phi_{\mathcal{L}}:V\rightarrow\mathbb{P}^{1}$ but $L_{2}$
instead of $L_{1}$, we obtain a birational morphism $\tau_{1}:V\rightarrow\mathbb{F}_{1}$.
The image of $L$ in $\mathbb{F}_{1}$ is a smooth $2$-section of
the $\mathbb{P}^{1}$-bundle structure $\pi_{1}:\mathbb{F}_{1}\rightarrow\mathbb{P}^{1}$
with self-intersection $4$. Letting $C_{0}$ be the exceptional section
of $\pi_{1}$ with self-intersection $\left(-1\right)$, one has necessarily
$L\sim2C_{0}+2F$ where $F$ is a fiber of $\pi_{1}$. So $L$ does
not intersect $C_{0}$ and its image in $\mathbb{P}^{2}$ by the morphism
$q:\mathbb{F}_{1}\rightarrow\mathbb{P}^{2}$ contracting $C_{0}$
is a smooth conic. Since $L_{1}$
is tangent to $L$ in $\mathbb{F}_{1}$, its image in $\mathbb{P}^{2}$
is the tangent to the image of $C$ at a point distinct from the one
blown-up by $q$. The rational pencil on $\mathbb{P}^{2}$ generated
by the image of $L$ and two times that of $L_{1}$ lifts to a rational
pencil $\overline{\rho}:V\dashrightarrow\mathbb{P}^{1}$ having $2L_{1}+L_{2}+2C_{0}$
and $L+\sum_{i=1}^{4}F_{i}$ as degenerate members. Its restriction
to $S=V\setminus D$ yields an $\mathbb{A}^{1}$-fibration $\rho:S\rightarrow\mathbb{P}^{1}$
with two degenerate fibers : one is irreducible of multiplicity two
consisting of the intersection of the proper transform of $C_{0}$
with $S$ and the other is reduced, consisting of the disjoint union
of the curves $F_{i}\cap S\simeq\mathbb{A}^{1}$, $i=1,\ldots,4$. 

\end{parn}
\begin{rem}
It follows from Corollary \ref{cor:Auto-desc} below that the automorphism
group of a smooth affine surface $S=V\setminus D$ as above is finite.
In particular, such surfaces are $\mathbb{A}^{1}$-ruled but do not
admit non trivial $\mathbb{G}_{a}$-actions and so the complement
of a smooth cubic surface $Q\subseteq \mathbb{P}^{3}$ is a smooth affine
$\mathbb{A}^{1}$-uniruled threefold without non trivial $\mathbb{G}_{a}$-action
which has the structure of a family $\rho_{t}(L):\mathbb{P}^{3}\setminus Q\rightarrow\mathbb{A}^{1}$
of $\mathbb{A}^{1}$-ruled affine surfaces without non trivial $\mathbb{G}_{a}$-actions.
In contrast, the total space of a family of affine varieties with
non trivial $\mathbb{G}_{a}$-actions often admits itself a non trivial
$\mathbb{G}_{a}$-action. For instance, suppose that $\varphi:X=\mathrm{Spec}(B)\rightarrow S=\mathrm{Spec}(A)$
is a dominant morphism between complex affine varieties and that $\partial:B\rightarrow B$
is an $A$-derivation of $B$ which is locally nilpotent in restriction
to general closed fibers of $\varphi$, defining non trivial $\mathbb{G}_{a}$-actions
on these fibers. Then $\partial$ is a locally nilpotent $A$-derivation
of $B$ and hence $X$ carries a non trivial $\mathbb{G}_{a}$-action.
Indeed, up to shrinking $S$ if necessary, we may assume that for
every maximal ideal $\mathfrak{m}\in\mathrm{Specmax}(A)$, the
$A/\mathfrak{m}$-derivation $\partial_{\mathfrak{m}}:B/\mathfrak{m}B\rightarrow B/\mathfrak{m}B$
induced by $\partial$ is locally nilpotent. Given an element $f\in B$
and an integer $n\geq0$, denote by $K^{n}(f)$ the closed sub-variety
of $S$ whose points are the maximal ideals $\mathfrak{m}\in\mathrm{\mathrm{Specmax}(A)}$
for which the residue class of $f$ in $B/\mathfrak{m}B$ belongs
to $\mathrm{Ker}\partial_{\mathfrak{m}}^{n}$. The hypothesis implies
that $S$ is equal to the increasing union of its closed sub-variety
$K^{n}(f)$, $n\geq0$, and so this sequence must stabilizes as the
base field $\mathbb{C}$ is uncountable. So there exists $n_{0}=n_{0}(f)$
such that the restriction of $\partial^{n_{0}}f$ to every closed
fiber of $\varphi:X\rightarrow S$ is the zero function, and hence $\partial^{n_{0}}f=0$
since these fibers form a dense subset of $X$. 
\end{rem}

\subsection{Automorphisms of general fibers of special rational pencils}\label{3.2} 

This sub-section is devoted to the study of automorphisms of general
fibers of the rational fibrations $\rho:\mathbb{P}^{3}\setminus Q\rightarrow\mathbb{A}^{1}$
constructed as above. We show in particular that such general fibers
admit biregular involutions induced by Geiser involutions of their
projective closures. 

\begin{parn} To simplify the discussion, let us call a pair $\left(V,D\right)$ as in 
\ref{par:Log-res} \emph{special} if $V$ is a smooth cubic hypersurface in ${\mathbb P}^3$ 
and $D$ is composed of three concurrent lines $L+L_{1}+L_{2}$ meeting in an Eckardt 
point $p$ of $V$, or of the union of $L$ and a smooth conic $C$
intersecting $L$ with multiplicity two in a single point $p$. 
Similarly
as in $\S$ \ref{par:Log-res}, we denote by $\alpha:W\rightarrow V$
the blow-up of $p$ and we denote by $E$ its exceptional divisor. 

By construction, $W$ is a weak del Pezzo surface (i.e., 
$-K_{W}$ is nef and big) of degree $2$  on which the anticanonical linear system
$|-K_{W}|$ defines a morphism $\theta:W\rightarrow\mathbb{P}^{2}$.
The latter factors into a birational morphism $W\rightarrow Y$ contracting
$L$ (resp. $L$, $L_{1}$ and $L_{2}$) if $D=L+C$ (resp. if $D=L+L_{1}+L_{2})$
followed by a Galois double covering $Y\rightarrow\mathbb{P}^{2}$
ramified over a quartic curve $\Delta$ with a unique double point
(resp. three double points) located at the image of $L$ (resp. at
the images of $L$, $L_{1}$ and $L_{2}$). 
The non trivial involution of the double covering $Y\rightarrow\mathbb{P}^{2}$
induces an involution $G_{W}:W\rightarrow W$ fixing $L$ and exchanging
the proper transform of $D$ and the exceptional divisor $E$ (resp.
fixing $L$, $L_{1}$ and $L_{2}$). The latter descends either to
birational involution $G_{V,p}:V\dashrightarrow V$ if $D=L+C$, or
to a biregular involution $G_{V,p}:V\rightarrow V$ having $p$ as
an isolated fixed point if $D=L+L_{1}+L_{2}$. 

In each case, we say that $G_{V,p}$ is the \emph{Geiser involution
of $V$ with center at $p$}. By construction, $G_{V,p}$ restricts
further to a biregular involution $j_{G_{V,p}}$ of the affine surface
$S=V\setminus D$ which we call the \emph{affine Geiser involution
of $S$ with center at $p$}. 

\end{parn}
\begin{prop}
\label{prop:Iso-desc} Let $\varphi:S=V\setminus D\stackrel{\sim}{\rightarrow}S'=V'\setminus D'$
be an isomorphism between affine surfaces associated to special pairs
$\left(V,D\right)$ and $\left(V',D'\right)$ respectively. Then the
following assertions hold:

a) If either $D$ or $D'$ consists of three concurrent lines then
$\varphi$ extends to an isomorphism of pairs $\overline{\varphi}:\left(V,D\right)\stackrel{\sim}{\rightarrow}(V',D')$.

b) Otherwise, if both $D=L+C$ and $D'=L'+C'$ consist of the union
of a line and  a conic, then the induced birational map $\overline{\varphi}:V\dashrightarrow V'$
is either an isomorphism of pairs, or it can be decomposed as $\overline{\varphi}=\Psi\circ G_{V,p}$
or $\overline{\varphi}=G_{V',p'}\circ\Psi'$ where $\Psi,\Psi'$ are
isomorphisms of pairs and where $G_{V,p}$ and $G_{V',p'}$ denote
the Geiser involutions of $V$ and $V'$ with centers are $p=L\cap C$
and $p'=L'\cap C'$ respectively. Furthermore, every birational map
$V\dashrightarrow V'$ of the form \textup{$\Psi\circ G_{V,p}$ (resp.
$G_{V',p'}\circ\Psi'$) can be uniquely re-written in the form $G_{V',p'}\circ\Psi'$
(resp. $\Psi\circ G_{V,p}$). }In particular, if there exists an isomorphism $\varphi:S=V\setminus D\stackrel{\sim}{\rightarrow}S'=V'\setminus D'$
then $V$ and $V'$ are isomorphic smooth cubic surfaces. \end{prop}
\begin{proof}
Let $\sigma:(\tilde{V},\tilde{D})\rightarrow(V,D)$ and $\sigma':(\tilde{V}',\tilde{D}')\rightarrow(V',D')$
be the minimal log-resolutions of the pairs $\left(V,D\right)$ and 
$(V',D')$ respectively described in $\S$ \ref{par:Log-res}. Since
$V\setminus D\simeq\tilde{V}\setminus\tilde{D}$ and $V'\setminus D'\simeq\tilde{V}'\setminus\tilde{D}'$
by construction, every isomorphism $\varphi:S\rightarrow S'$ extends
in a natural way to a birational map $\tilde{\varphi}:\tilde{V}\dashrightarrow\tilde{V}'$
restricting to $\varphi$ on $S$. We claim that $\tilde{\varphi}$
is in fact a biregular isomorphism of pairs. Indeed, suppose on the
contrary that $\tilde{\varphi}$ is strictly birational and let $\tilde{V}\stackrel{\beta}{\leftarrow}X\stackrel{\beta'}{\rightarrow}\tilde{V}'$
be its minimal resolution. Recall that the minimality of the resolution
implies in particular that there is no $\left(-1\right)$-curve in
$X$ which is exceptional for $\beta$ and $\beta'$ simultaneously. 
Since $\tilde{V}$ is smooth and $\tilde{D}$ is an SNC divisor, $\beta'$
decomposes into a finite sequence of blow-downs of successive $\left(-1\right)$-curves
supported on the boundary $B=\beta^{-1}(\tilde{D})_{\mathrm{red}}=(\beta')^{-1}(\tilde{D})_{\mathrm{red}}$.
In view of the structure of $\tilde{D}$, the only possible $\left(-1\right)$-curve
in $B$ which is not exceptional for $\beta$ is the proper transform
of $E$ if $D=L+L_{1}+L_{2}$ or the proper transform of $E_{2}$
if $D=L+C$. But after the contraction of these curves, the boundary
would no longer be an SNC divisor, a contradiction. This implies that
$\tilde{\varphi}:\tilde{V}\rightarrow\tilde{V}'$ is a morphism and
the same argument shows that it does not contract any curve in the
boundary $\tilde{D}$. So $\tilde{\varphi}:\tilde{V}\rightarrow\tilde{V}'$
is an isomorphism restricting to an isomorphism between $\tilde{V}\setminus\tilde{D}$
and $\tilde{V}'\setminus\tilde{D}'$ whence an isomorphism between
the pairs $(\tilde{V},\tilde{D})$ and $(\tilde{V}',\tilde{D}')$.
It follows in particular that the intersection matrices of the divisors
$\tilde{D}$ and $\tilde{D}'$ must be the same up to a permutation.
In view of the description given in $\S$ \ref{par:Log-res}, we conclude
that either $D$ and $D'$ simultaneously consist of three concurrent
lines or of the union of a line and a smooth conic. 

In the first case, since the exceptional divisors $E$ of the blow-up
of $V$ at $p$ and $E'$ of the blow-up of $V'$ at $p'$ are the
only $\left(-1\right)$-curves in $\tilde{D}$ and $\tilde{D}'$ respectively,
the biregular isomorphism $\tilde{\varphi}:\tilde{V}\rightarrow\tilde{V}'$
extending $\varphi$ necessarily maps $E$ isomorphically onto $E'$.
So $\tilde{\varphi}$ descends to an isomorphism of pairs $\overline{\varphi}:(V,D)\rightarrow(V',D')$
which gives a). 

In the second case, letting $D=L+C$ and $D'=L'+C'$, a similar argument
implies that the isomorphism $\tilde{\varphi}:\tilde{V}\rightarrow\tilde{V}'$
maps $E_{2}$ onto $E_{2}'$ and the proper transform of $L$ onto
that of $L'$. This implies in turn that $\tilde{\varphi}$ descends
to an isomorphism $\tilde{\varphi}_{1}:W\rightarrow W'$ between the
surfaces obtained from $V$ and $V'$ by blowing-up the points
$p=L\cap C$ and $p'=L'\cap C'$ with respective exceptional divisors
$E$ and $E'$. Now we have the following alternative: either $\tilde{\varphi}_{1}$
maps $E$ and $C$ onto $E'$ and $C'$ respectively, and then it descends to a biregular
isomorphism of pairs $\overline{\varphi}:(V,D)\rightarrow(V',D')$, 
or $\tilde{\varphi}_{1}$ maps $E$ and $C$ onto $C'$ and $E'$ respectively.
In the second case, by composing $\overline{\varphi}:V\dashrightarrow V'$
either by $G_{V,p}$ on the left or by $G_{V',p'}$ on the right,
we get a new birational maps $\Psi,\Psi':V\dashrightarrow V'$ which
lift to a birational map $\Psi_{1}:W\dashrightarrow W'$ mapping
$L$, $E$ and $C$ isomorphically onto $L'$, $E'$ and $C'$ respectively.
The previous discussion implies that $\Psi$ and $\Psi'$ are isomorphism
of pairs and so we get the desired decompositions $\overline{\varphi}=\Psi\circ G_{V,p}^{-1}=\Psi\circ G_{V,p}$
and $\overline{\varphi}=G_{V',p'}^{-1}\circ\Psi'=G_{V',p'}\circ\Psi'$.
The last assertion about the uniqueness of the re-writing is clear
by construction.\end{proof}
\begin{cor}
\label{cor:Auto-desc} Let $(V,D)$ be a special pair and let $S=V\setminus D$
be the corresponding affine surface. Then the following holds:

a) If $D$ consists of three concurrent lines then $\mathrm{Aut}(S)$
is a nontrivial finite subgroup of $\mathrm{PGL}(4,\mathbb{C})$ which
coincides with the automorphism group $\mathrm{Aut}(V,D)$ of the
pair $\left(V,D\right)$. 

b) Otherwise, if $D$ consists of a line and a smooth conic, then we
have an exact sequence 
\[
0\rightarrow\mathrm{Aut}(V,D)\rightarrow\mathrm{Aut}(S)\rightarrow\mathbb{Z}_{2}\cdot j_{G_{V,p}}\rightarrow0
\]
and an isomorphism $\mathrm{Aut}(S)=\mathrm{Aut}(V,D)\rtimes\mathbb{Z}_{2}\cdot j_{G_{V,p}}$, 
where the affine Geiser involution $j_{G_{V,p}}$ acts on $\mathrm{Aut}(V,D)$
by $j_{G_{V,p}}\cdot\Psi=\Psi'$ where for every $\Psi\in\mathrm{Aut}(V,D)$,
$\Psi'$ is the unique automorphism of the pair $(V,D)$ such that
$\Psi\circ G_{V,p}=G_{V,p}\circ\Psi'$. \end{cor}
\begin{proof}
Since $V\subseteq \mathbb{P}^{3}$ is embedded by its anti-canonical
linear system $|-K_{V}|$, every automorphism of $V$ is induced by
a linear transformation of $\mathbb{P}^{3}$. Furthermore, since $D$
is a hyperplane section of $V$, $\mathrm{Aut}(V,D)$ is an algebraic
sub-group of the group of linear transformation of $\mathbb{P}^{3}$
preserving globally the corresponding hyperplane. The automorphism
group of a smooth cubic surface in $\mathbb{P}^{3}$ being finite,
it follows that $\mathrm{Aut}(V,D)$ is a finite algebraic group.
If $D$ consists of three concurrent lines, then their common point
$p$ is an Eckardt point of $V$ which is always the isolated fixed
point of a biregular involution of $V$ preserving $D$. This shows
that $\mathrm{Aut}(V,D)$ is never trivial in this case. In case b), 
where $D=L+C$ is the union of a line and smooth conic, the assertion
is an immediate consequence of the previous Proposition \ref{prop:Iso-desc}. 
\end{proof}
\begin{rem}
If $D=L+C$ then every automorphism $\Psi$ of the pair $\left(V,D\right)$
preserves the rational pencil $\overline{q}:V\dashrightarrow\mathbb{P}^{1}$
constructed in $\S$ \ref{Rk:A1-fib}. Indeed such an automorphism
certainly preserves the conic $C$ and it maps the five $\left(-1\right)$-curves
$F_{1},\ldots,F_{5}$ to five other $\left(-1\right)$-curves each
intersecting $C$ transversally. But it follows from
the construction of $V$ from $\mathbb{P}^{2}$ that $F_{1},\ldots,F_{5}$
are the only five $\left(-1\right)$-curves in $V$ with this property.
Similarly, since $L$ and $p=L\cap C$ are $\Psi$-invariant, the
image of $T$ by $\Psi$ must be a smooth rational curve intersecting
$L$ and $C$ transversally at $p$ but the construction of $V$ again
implies that $T$ is the only curve in $V$ with this property. It
follows that the divisors $C+\sum_{i=1}^{5}F_{i}$ and $3T+L$ which
generate the pencil $\overline{q}:V\dashrightarrow\mathbb{P}^{1}$
are $\Psi$-invariant whence that $\overline{q}$ is globally preserved
by $\Psi$ as claimed. This implies in turn that every automorphism
$\varphi$ of $S$ which extends to a biregular automorphism of the
pair $\left(V,D\right)$ fits into a commutative diagram 
\begin{eqnarray*}
S & \stackrel{\varphi}{\rightarrow} & S\\
q\downarrow &  & \downarrow q\\
\mathbb{P}^{1} & \stackrel{\xi}{\rightarrow} & \mathbb{P}^{1}, 
\end{eqnarray*}
where $\xi\in\mathrm{PG}L(2,\mathbb{C})$ fixes the two points of
$\mathbb{P}^{1}$ corresponding to the non isomorphic degenerate fibers
$3T\cap S$ and $\sum_{i=1}^{5}F_{i}\cap S$ of $q$. 
In contrast, it is straightforward to check from the construction
that $q:S\rightarrow\mathbb{P}^{1}$ is not globally preserved by
the affine Geiser involution $j_{G_{V,p}}$ of $S$ with center at
$p=D\cap L$ and hence that $S$ carries a second $\mathbb{A}^{1}$-fibration
$q\circ j_{G_{V,p}}:S\rightarrow\mathbb{P}^{1}$ whose general fibers
are distinct from that of $q$. In particular, if $V$ is chosen generally 
so that $\mathrm{Aut}(V,D)$ is trivial, then $\mathrm{Aut}(S)$ is
isomorphic to $\mathbb{Z}_{2}$, generated by the affine Geiser involution
$j_{G_{V,p}}$ (cf. Corollary \ref{cor:Auto-desc}, (b)), 
and $S$ carries two conjugated $\mathbb{A}^{1}$-fibrations
$q:S\rightarrow\mathbb{P}^{1}$ and $q\circ j_{G_{V,p}}:S\rightarrow\mathbb{P}^{1}$,
each of these having no non trivial automorphism.
\end{rem}

\subsection{Automorphisms of special rational pencils }

\indent\newline\noindent This subsection is devoted to the proof
of the following result:
\begin{prop}
Let $Q\subseteq \mathbb{P}^{3}$ be a smooth cubic surface and let $\overline{\rho}:\mathbb{P}^{3}\dashrightarrow\mathbb{P}^{1}$
be the special pencil associated to a line $L\subseteq Q$ and a distinguished
hyperplane $H\in\left|\mathcal{O}_{\mathbb{P}^{3}}(1)\otimes\mathcal{I}_{L}\right|$
as in $\S$ \ref{par:Log-res}. Then every automorphism of the induced
rational fibration $\rho:\mathbb{P}^{3}\setminus Q\rightarrow\mathbb{A}^{1}$
is the restriction of an automorphism of $\mathbb{P}^{3}$. \end{prop}
\begin{proof}
Let us denote by $\mathrm{Aut}(\mathbb{P}^{3}\setminus Q,\rho)$ the
subgroup of $\mathrm{Aut}(\mathbb{P}^{3}\setminus Q)$ consisting
of automorphisms preserving the fibration $\rho:\mathbb{P}^{3}\setminus Q\rightarrow\mathbb{A}^{1}=\mathrm{Spec}(\mathbb{C}[\lambda])$
fiberwise. Taking restriction over the generic point $\eta$ of $\mathbb{A}^{1}$
induces an injective homomorphism from $\mathrm{Aut}(\mathbb{P}^{3}\setminus Q,\rho)$
to the group $\mathrm{Aut}(S_{\eta})$ of automorphisms of the generic
fiber $S_{\eta}$ of $\rho$, and we identify from now on $\mathrm{Aut}(\mathbb{P}^{3}\setminus Q,\rho)$
with its image in $\mathrm{Aut}(S_{\eta})$. By definition, $S_{\eta}$
is a nonsingular affine surface defined over the field $\mathbb{C}(\lambda)$
whose closure in $\mathbb{P}_{\mathbb{C}(\lambda)}^{3}$ is a nonsingular
cubic surface $V_{\eta}$ such that $D_{\eta}=V_{\eta}\setminus S_{\eta}$
consists of either three lines $L_{\eta}$, $L_{\eta,1}$ and $L_{\eta,2}$
meeting a unique $\mathbb{C}(\lambda)$-rational point or the union
of a line $L_{\eta}$ and nonsingular conic $C_{\eta}$ intersecting
$L_{\eta}$ at a unique $\mathbb{C}(\lambda)$-rational point $p$.
The same argument as in the proofs of Proposition \ref{prop:Iso-desc}
and its Corollary \ref{cor:Auto-desc} implies that $\mathrm{Aut}(S_{\eta})=\mathrm{Aut}(V_{\eta},D_{\eta})$
in the first case whereas $\mathrm{Aut}(S_{\eta})=\mathrm{Aut}(V_{\eta},D_{\eta})\rtimes\mathbb{Z}_{2}\cdot j_{G_{V_{\eta},p}}$, 
where $j_{G_{V_{\eta},p}}$ denotes the affine Geiser involution with 
the center at $p$, in the second case. In the first case, we conclude that every automorphism of $\mathrm{Aut}(\mathbb{P}^{3}\setminus Q,\rho)$
is generically of degree one when considered as a birational self-map
of $\mathbb{P}^{3}$ and hence is the restriction of an automorphism
of $\mathbb{P}^{3}$. So it remains to show that in the second case,
one has $\mathrm{Aut}(\mathbb{P}^{3}\setminus Q,\rho)\subseteq \mathrm{Aut}(V_{\eta},D_{\eta})$
necessarily. Suppose on the contrary that $\mathrm{Aut}(\mathbb{P}^{3}\setminus Q,\rho)\not\subseteq \mathrm{Aut}(V_{\eta},D_{\eta})$.
Then since we have an extension 
\[
0\rightarrow\mathrm{Aut}(V_{\eta},D_{\eta})\rightarrow\mathrm{Aut}(S_{\eta})\rightarrow\mathbb{Z}_{2}\cdot j_{G_{V_{\eta,p}}}\rightarrow0,
\]
it would follow that $j_{G_{V_{\eta,p}}}\in\mathrm{Aut}(\mathbb{P}^{3}\setminus Q,\rho)$.
On the other hand, letting $\pi_{p}:\mathbb{P}^{3}\dashrightarrow\mathbb{P}^{2}$
be the projection from the point $p$, the rational map $(\overline{\rho},\pi_{p}):\mathbb{P}^{3}\dashrightarrow\mathbb{P}^{1}\times\mathbb{P}^{2}$
is generically of degree $2$, contracting the plane $H$ to a line
$\pi_{p}(H)\in\overline{\rho}(H)\times\mathbb{P}^{2}\simeq\mathbb{P}^{2}$
and inducing a Galois double covering 
\[
\mathbb{P}^{3} \setminus (Q\cup H) \rightarrow\mathbb{P}^{1}\setminus\left\{ \overline{\rho}(Q)\cup\overline{\rho}(H)\right\} \times\mathbb{P}^{2}\setminus(\pi_{p}(Q)\cup\pi_{p}(H))\simeq\mathbb{A}_{*}^{1}\times\mathbb{A}^{2}\setminus(\pi_{p}(Q)).
\]
By definition, the affine Geiser involution $j_{G_{V_{\eta},p}}$
of $S_{\eta}$ is simply the restriction to the generic fiber of $\rho$
of the nontrivial involution $J$ of this covering. So $j_{G_{V_{\eta,p}}}\in\mathrm{Aut}(\mathbb{P}^{3}\setminus Q,\rho)$
considered as a birational self-map of $\mathbb{P}^{3}$ would coincide
with $J$ which is absurd since the latter contracts $H$ whence the
fiber $\rho^{-1}(\overline{\rho}(H))$ of $\rho$. Thus $\mathrm{Aut}(\mathbb{P}^{3}\setminus Q,\rho)\subseteq\mathrm{Aut}(V_{\eta},D_{\eta})$
as desired. 
\end{proof}

\bibliographystyle{amsplain}
 
\end{document}